\documentclass[amstex,12pt]{article}
\usepackage{amssymb,amsmath,amsthm}
\newtheorem{definition}{Definition}[section]
\newtheorem{theorem}{Theorem}[section]
\newtheorem{corollary}{Corollary}[section]
\newtheorem{lemma}{Lemma}[section]
\newtheorem{remark}{Remark}[section]
\newtheorem{example}{Example}[section]

\newcommand{\beq}{\begin{equation}}
\newcommand{\eeq}{\end{equation}}
\newcommand{\beqn}{\begin{eqnarray}}
\newcommand{\eeqn}{\end{eqnarray}}

\allowdisplaybreaks
\begin{document}
\title{Weighted pseudo-almost periodic functions on time scales with applications to cellular neural networks with  discrete delays\thanks{This work is supported by the National Natural
Sciences Foundation of People's Republic of China under Grant
11361072.}}
\author{ Yongkun Li\thanks{The corresponding author.  Email: yklie@ynu.edu.cn.} and Lili Zhao\\
Department of Mathematics,
Yunnan University\\
Kunming, Yunnan 650091\\
People's Republic of China}
\date{}
\maketitle
\begin{abstract}
In this paper, we first propose a concept of weighted pseudo-almost periodic functions on time scales
and   study some basic properties of weighted pseudo-almost periodic functions on time scales.   Then,
we establish some results about the existence of weighted pseudo-almost periodic solutions to linear dynamic
equations on time scales. Finally, as an application of our results, we  study the existence
and global  exponential stability of weighted pseudo-almost periodic solutions for a class of cellular neural networks
with discrete delays on time scales. The results of this paper are completely new.
\end{abstract}

\section{Introduction}
\setcounter{equation}{0} \vspace{1ex} \indent

The concept of pseudo-almost periodicity, which is a natural generalization of the
notion of almost periodicity, was introduced in the literature more than a decade
ago by C. Zhang [1-3]. Since its introduction in the literature, the notion of
pseudo-almost periodicity has generated several developments and extensions, see,
e.g., \cite{w1,w2}. The concept of weighted pseudo-almost periodicity
introduced in the literature in 2006 by Diagana \cite{w2}.  The notion of weighted
pseudo-almost periodicity has generated several developments since then, see for instance the
 references [6-8].
 Among other things, it has been utilized to study the qualitative
behavior to various differential and partial differential equations involving weighted pseudo-almost
periodic coefficients, see, e.g., [9-11].

On the other hand,
dynamic equations on time scales offer a new direction
in the study of dynamic systems which involve
differential equations and difference equations as special cases.
Their origin is connected with Stefan Hilger's work \cite{hiphd,hi1}.
In 1988, Stefan Hilger introduced the definition of
a $\Delta$-derivative.
The common derivative and the common forward difference
are special cases of the $\Delta$-derivative.
  In fact, the progressive field of
dynamic equations on time scales contains links to and extends the classical theory of differential
and difference equations. For instance, by choosing the time scale to be the set
of real numbers, the general result yields a result for differential equations. In a similar
way, by choosing the time scale to be the set of integers, the same general result yields a
result for difference equations. However, since there are many other time scales than just
the set of real numbers or the set of integers, one has a much more general result. For
these reasons, based on the concept of almost periodic time scales proposed in \cite{19,20},
the concept of pseudo-almost periodic functions on almost periodic time
scales was formally introduced by Li and Wang (2012) in \cite{23}. Moreover, some first results
were proven which concern the   pseudo-almost periodic  solution
to   dynamic equations on time scales.
However, to the best of our knowledge, there is no  concept of weighted   pseudo-almost
periodic functions on time scales   yet, so up to now, there was no work on discussing weighted   pseudo-almost periodic
problems of   dynamic equations on time scales before.

Also,
it is well known that cellular neural networks have been extensively
applied in areas of signal processing, image processing, pattern
recognition, optimization and associative memories. Since all these
applications closely relate to the dynamics, the dynamical behaviors
of cellular neural networks have been widely investigated. There
have been extensive results on the problem of the existence and
stability of equilibrium points, periodic solutions  and almost periodic solutions of cellular neural networks in the
literature. We refer the reader to [17-27] and the references cited
therein. However, to the best of our knowledge, few
authors have studied the problems of  pseudo-almost periodic
and weighted pseudo-almost periodic solutions of   neural networks \cite{cw1,cw2}. Moreover, it is known that the
existence and stability of almost periodic solutions play a
key role in characterizing the behavior of dynamical systems (see
[30-36]) and pseudo-almost periodicity is a natural generalization of the notion of almost periodicity.
Furthermore, because of the weights
involved, the concept of weighted pseudo-almost periodicity is more general and
richer than the concept of pseudo-almost periodicity.

Motivated by the above discussion, in this paper, we first introduce the concept of weighted pseudo-almost periodic functions on time scales
and study some basic properties of weighted pseudo-almost periodic functions on time scales. Then,
we establish some results about the existence of weighted pseudo-almost periodic solutions to linear dynamic
equations on time scales. Finally,
  as an application of our results, we  study the existence
and global  exponential stability of weighted pseudo-almost periodic solutions for the following   cellular neural network
with discrete delays on time scales
\begin{equation}\label{e1}
x_i^{\Delta}(t)=-c_i(t)x_i(t)+\sum^n_{j=1}a_{ij}(t)f_j(x_j(t))+\sum^n_{j=1}b_{ij}(t)f_j(x_j(t-\gamma_{ij}))+I_i(t),\,\,\,t\in\mathbb{T},
\end{equation}
where $i=1,2,\ldots,n$ and $\mathbb{T}$ is an almost periodic time scale which will be
defined in the next section, $x_i(t)$ correspond to
the activations of the $i$th neurons at the time $t$, $c_i(t)$ are positive
functions, they denote the rate with which the cell $i$
reset their potential to the resting state when isolated from the
other cells and inputs at time $t$, $a_{ij}(t)$ and $b_{ij}(t)$ are
the connection weights at time $t$, $\gamma_{ij}$
are nonnegative, which corresponds to the finite speed of the axonal
signal transmission, $I_i(t)$ denote the external
inputs at time $t$, $f_i$ are the activation functions of
signal transmission. For each interval $J$ of $\mathbb{R},$ we
denote by $J_\mathbb{T}=J\cap\mathbb{T}$.

The system \eqref{e1} is supplemented with the initial values given by
\[
x_i(s)=\varphi_i(s),\,\,i=1,2,\ldots,n,
\]
where $\varphi_i(\cdot)$ denotes a real-value bounded
right-dense continuous function defined on $[-\gamma,0]_{\mathbb{T}}$, and
$
\gamma_i=\max_{1\leq j\leq
n}\{\gamma_{ij}\}, \gamma=\max_{1\leq i\leq
n}\{\gamma_i\}
$.

Throughout this paper, we assume that the following conditions hold:
\begin{itemize}
\item [$(H_1)$] $f_j\in C(\mathbb{R},\mathbb{R})$ and there
    exist positive constants
    $\alpha_j$ such that
    \[
      |f_j(u)-f_j(v)|\leq \alpha_j|u-v|,\quad u,v\in\mathbb{R},\,\,j=1,2,\ldots,n;
    \]
\item [$(H_2)$] $c_i,a_{ij},b_{ij}$
are almost periodic functions on
$\mathbb{T}$, where $i,j=1,2,\ldots,n$;
\item [$(H_3)$]
$\inf\limits_{t\in\mathbb{T}}c_i(t)>0,\,\,
-c_i\in \mathcal{R}^+,\gamma_{ij}\in\Pi,\,i,j=1,2,\ldots,n$,
where $\mathcal{R}^+,\,\Pi$ will be defined in the next section;
\item [$(H_4)$] For fixed $u\in\mathbb{U}_{\infty}^{Inv}$. $I_i\,(i=1,2,\ldots,n)$ are weighted pseudo-almost periodic functions,
 where $\mathbb{U}_{\infty}^{Inv}$ will be defined in the next section.
\end{itemize}

\begin{remark}
If   $\mathbb{T} = \mathbb{R}$, then \eqref{e1} reduces to the following form
\begin{equation}\label{ae12}
x_i'(t)=-c_i(t)x_i(t)+\sum^n_{j=1}a_{ij}(t)f_j(x_j(t))+\sum^n_{j=1}b_{ij}(t)f_j(x_j(t-\gamma_{ij}))+I_i(t),\,\,i=1,2,\ldots,n,\,\,t\in\mathbb{R},
\end{equation}
if  $\mathbb{T} = \mathbb{Z}$, then \eqref{e1} reduces to the following form
\begin{eqnarray}\label{ae13}
x_i(k+1)-x_i(k)&=&-c_i(k)x_i(t)+\sum^n_{j=1}a_{ij}(k)f_j(x_j(k))+\sum^n_{j=1}b_{ij}(t)f_j(x_j(k-\gamma_{ij}))\nonumber\\
&&+I_i(k),i=1,2,\ldots,n,\,\,k\in\mathbb{Z}.
\end{eqnarray}
To the best of our knowledge, there is no paper published on the existence and exponential
stability of  weighted pseudo-almost periodic solutions for \eqref{ae12} and \eqref{ae13}.
\end{remark}

The organization of the rest of this paper is as follows. In Section
2, we introduce some definitions and make some preparations for
later sections. In Section 3, we propose a concept of weighted pseudo-almost periodic functions on almost periodic time scales and study some their basic properties.
In Section 4, we study the existence of weighted pseudo-almost periodic solutions to linear dynamic equations on time scales.
In Section 5 and Section 6, based on the results obtained in the previous sections,   Banach's fixed
point theorem and $\Delta$-differential inequalities on time scales, we present
some sufficient conditions which guarantee the existence and global exponential  stability of weighted pseudo-almost periodic solutions to \eqref{e1}.
 In Section 7, we present examples to illustrate
the feasibility and effectiveness of our results obtained in Section 5 and Section 6.

\section{Preliminaries}\setcounter{equation}{0}
\vspace{1ex}
\indent

In  this section, we shall first recall some basic definitions and
prove some lemmas.

Let $\mathbb{T}$ be a nonempty closed subset (time scale) of
$\mathbb{R}$. The forward and backward jump operators $\sigma,
\rho:\mathbb{T}\rightarrow\mathbb{T}$ and the graininess
$\mu:\mathbb{T}\rightarrow\mathbb{R}^+$ are defined, respectively,
by
\[
\sigma(t)=\inf\{s\in\mathbb{T}:s>t\},\,\,\,\,
\rho(t)=\sup\{s\in\mathbb{T}:s<t\}\,\,\,\,{\rm
and}\,\,\,\,\mu(t)=\sigma(t)-t.
\]

A point $t\in\mathbb{T}$ is called left-dense if $t>\inf\mathbb{T}$
and $\rho(t)=t$, left-scattered if $\rho(t)<t$, right-dense if
$t<\sup\mathbb{T}$ and $\sigma(t)=t$, and right-scattered if
$\sigma(t)>t$. If $\mathbb{T}$ has a left-scattered maximum $m$,
then $\mathbb{T}^k=\mathbb{T}\setminus\{m\}$; otherwise
$\mathbb{T}^k=\mathbb{T}$. If $\mathbb{T}$ has a right-scattered
minimum $m$, then $\mathbb{T}_k=\mathbb{T}\setminus\{m\}$; otherwise
$\mathbb{T}_k=\mathbb{T}$.

A function $f:\mathbb{T}\rightarrow\mathbb{R}$ is right-dense
continuous provided it is continuous at right-dense point in
$\mathbb{T}$ and its left-side limits exist at left-dense points in
$\mathbb{T}$. If $f$ is continuous at each right-dense point and
each left-dense point, then $f$ is said to be continuous function on
$\mathbb{T}$.

For $y:\mathbb{T}\rightarrow\mathbb{R}$ and $t\in\mathbb{T}^k$, we
define the delta derivative of $y(t)$, $y^\Delta(t)$, to be the
number (if it exists) with the property that for a given
$\varepsilon>0$, there exists a neighborhood $U$ of $t$ such that
\[
|[y(\sigma(t))-y(s)]-y^\Delta(t)[\sigma(t)-s]|<\varepsilon|\sigma(t)-s|
\]
for all $s\in U$.

If $y$ is continuous, then $y$ is right-dense continuous, and if $y$
is delta differentiable at $t$, then $y$ is continuous at $t$.

Let $y$ be right-dense continuous. If $Y^{\Delta}(t)=y(t)$, then we
define the delta integral by
$\int_a^{t}y(s)\Delta s=Y(t)-Y(a).$

A function $r:\mathbb{T}\rightarrow\mathbb{R}$ is called regressive
if
$
1+\mu(t)r(t)\neq 0
$
for all $t\in \mathbb{T}^k$. The set of all regressive and
right-dense continuous functions $r:\mathbb{T}\rightarrow\mathbb{R}$ will
be denoted by $\mathcal{R}=\mathcal{R}(\mathbb{T})=\mathcal{R}(\mathbb{T},\mathbb{R})$. We
define the set
$\mathcal{R}^+=\mathcal{R}^+(\mathbb{T},\mathbb{R})=\{r\in \mathcal{R}:1+\mu(t)r(t)>0,\,\,\forall
t\in\mathbb{T}\}$.

If $r$ is regressive function, then the generalized exponential
function $e_r$ is defined by
\[
e_r(t,s)=\exp\bigg\{\int_s^t
\xi_{\mu(\tau)}(r(\tau))\Delta\tau\bigg\},\,\,\,{\rm for}\,s,t\in
\mathbb{T},
\]
with the cylinder transformation
\[
\xi_h(z)=\bigg\{\begin{array}{ll} \frac{\mathrm{Log}(1+hz)}{h} &{\rm
if}\,h\neq
0,\\
z &{\rm if}\,h=0.\\
\end{array}
\]
Let $p,q:\mathbb{T}\rightarrow\mathbb{R}$ be two regressive
functions, we define
\[
p\oplus q:=p+q+\mu pq,\,\,\,\,\ominus p:=-\frac{p}{1+\mu
p},\,\,\,\,p\ominus q:=p\oplus(\ominus q).
\]
Then the generalized exponential function has the following
properties.

\begin{lemma}\cite{14,15}
 Assume that $p,q:\mathbb{T}\rightarrow\mathbb{R}$ are two
regressive functions, then
\begin{itemize}
    \item  [$(i)$]  $e_0(t,s)\equiv 1$ and $e_p(t,t)\equiv 1$;
    \item  [$(ii)$] $e_p(\sigma(t),s)=(1+\mu(t)p(t))e_p(t,s)$;
    \item  [$(iii)$]$e_p(t,\sigma(s))=\frac{e_p(t,s)}{1+\mu(s)p(s)}$;
    \item  [$(iv)$] $\frac{1}{e_p(t,s)}=e_{\ominus p}(t,s)$;
    \item  [$(v)$] $(e_{\ominus p}(t,s))^{\Delta}=(\ominus p)(t)e_{\ominus
    p}(t,s)$.
\end{itemize}
\end{lemma}
\begin{definition} \cite{21}
For every $x,y\in\mathbb{R}$, $[x,y)=\{t\in\mathbb{R}:x\leq t<y\}$, define a countably additive measure $m_1$ on the set
$
\mathfrak{F_1}=\{[\tilde{a},\tilde{b})\cap\mathbb{T}:\tilde{a},\tilde{b}\in\mathbb{T},\tilde{a}\leq\tilde{b}\},
$
that assigns to each interval $[\tilde{a},\tilde{b})\cap\mathbb{T}$ its length, that is
$
m_1([\tilde{a},\tilde{b})=\tilde{b}-\tilde{a}.
$
The interval $[\tilde{a},\tilde{a})$ is understood as the empty set. Using $m_1$, it generates the outer measure $m_1^{*}$
on $\textrm{P}(\mathbb{T})$, defined for each $E\in\textrm{P}(\mathbb{T})$ as
 {\setlength\arraycolsep{2pt}
\begin{eqnarray*}
m_1^{*}(E)=\left\{\begin{array}{lll}
\inf_{\tilde{\Re}}\big\{\sum_{i\in I_{\tilde{\Re}}}(\tilde{b_i}-\tilde{a_i})\big\}\in\mathbb{R}^+,\,\,b\in \mathbb{T}\setminus E,\\
+\infty, \,\,\,\,\,\,\,\,\,\,\,\,\,\,\,\,\,\,\,\,\,\,\,\,\,\,\,\,\,\,\,\,\,\,\,\,\,\,\,\,\,\,\,\,\,\,\,\,\,\,\,\,\,\,\,\,\,\,\,\,\,\,b\in E,
\end{array}\right.
\end{eqnarray*}}
with
\[
\tilde{\Re}=\bigg\{\{[\tilde{a_i},\tilde{b_i})\cap\mathbb{T}\in\mathfrak{F_1}\}_{i\in I_{\tilde{\Re}}}:I_{\tilde{\Re}}\subset
\mathbb{N},E\subset\bigcup_{i\in I_{\tilde{\Re}}}([a_i,b_i)\cap\mathbb{T})\bigg\}.
\]
A set $\wedge\subset\mathbb{T}$ is said to be $\Delta$-measurable if the following equality:
\[
m_1^{*}(E)=m_1^{*}(E\cap\wedge)+m_1^{*}(E\cap(\mathbb{T}\backslash\wedge))
\]
holds true for all subset $E$ of $\mathbb{T}$. Define the family
$
\mathcal{M}(m_1^{*})=\{\wedge\subset\mathbb{T}:\wedge$ is $\Delta$-measurable$\}$, the Lebesgue $\Delta$-measure, denoted by
$\mu_{\Delta}$ is the restriction of $m_1^{*}$ to $\mathcal{M}(m_1^{*})$.
\end{definition}
\begin{definition} \cite{21}
We say that $f:\mathbb{T}\rightarrow\bar{\mathbb{R}}\equiv[-\infty,+\infty]$ is $\Delta$-measurable if for every $\alpha\in\mathbb{R}$, the set
$
f^{-1}([-\infty,\alpha))=\{t\in\mathbb{T}:f(t)<\alpha\}
$
is $\Delta$-measurable.
\end{definition}
\begin{lemma} \cite{21}\label{lem22}
Let $\wedge\subset\mathbb{T}$. Then $\wedge$ is $\Delta$-measurable if and only if $\wedge$ is Lebesgue measurable.
\end{lemma}
By using Lemma \ref{lem22}, we can get the following two corollaries:
\begin{corollary}\label{cl21}
If $A$ is a closed subset of $\mathbb{R}$, then $A\cap\mathbb{T}$ is $\Delta$-measurable.
\end{corollary}
\begin{corollary}\label{cl22}
If $f\in C(\mathbb{T},\mathbb{R})$, then $f$ is $\Delta$-measurable.
\end{corollary}
\begin{theorem} \cite{21}\label{thm21}
Let $E\subset\mathbb{T}$ be a $\Delta$-measurable set and let $(f_m)_{m\in\mathbb{N}}$ be a sequence of $\Delta$-measurable functions such that
for every $t\in\mathbb{T}$ the following conditions are satisfied:
\begin{itemize}
\item [$(a)$] $0\leq f_m(t)\leq f_{m+1}(t)\leq\infty$ for all $m\in\mathbb{N}$;
\item [$(b)$] $\lim\limits_{m\rightarrow \infty}f_m(t)=f(t)$.
\end{itemize}
Then $f$ is $\Delta$-measurable and
$
\lim_{m\rightarrow\infty}\int_{E}f_m(s)\Delta s=\int_{E}f(s)\Delta s.
$
\end{theorem}

\begin{definition} \cite{19,20}
A time scale is called an almost periodic time scale if
\[
\Pi:=\{\tau\in\mathbb{R}:t\pm\tau\in\mathbb{T},\forall
t\in\mathbb{T}\}\neq\{0\}.
\]
\end{definition}

In this paper, we restrict our discussions on almost periodic time scales.

\begin{lemma}\cite{m1}\label{lem23}
If $\mathbb{T}$ is an almost periodic time scale and $\tau\in\Pi$, then $\sigma(t+\tau)=\sigma(t)+\tau$ for  $t\in\mathbb{T}$.
\end{lemma}
\begin{corollary}
If $\mathbb{T}$ is an almost periodic time scale, then $\mu(t+\tau)=\mu(t),\,\forall t\in\mathbb{T},\,\tau\in\Pi$.
\end{corollary}

\begin{lemma}\label{lem24}
Let $\mathbb{T}$ be an almost periodic time scale. If $f:\mathbb{T}\rightarrow\mathbb{R}$ is right-dense continuous, then
\[
\int_a^{b}f(t+\tau)\Delta t=\int_{a+\tau}^{b+\tau}f(t)\Delta t,
\]
where $a,b\in\mathbb{T},\,\tau\in \Pi$.
\end{lemma}
\begin{proof}
Let $F(t)$ be an antiderivative of $f(t)$ and $G(t):=F(t+\tau)$ for all $t\in\mathbb{T}$. Since $t$ is  right-scattered or right-dense if and only if $t+\tau$ are right-scattered
  or right-dense, respectively, $G(t)$ is a continuous function.

  $\mathbf{Case}$ (1): If $t$ is right-scattered, then $G(t)$ is differentiable
at $t$ and
\[
G^{\Delta}(t)=\frac{G(\sigma(t))-G(t)}{\mu(t)}=\frac{F(\sigma(t)+\tau)-F(t+\tau)}{\mu(t+\tau)}=\frac{F(\sigma(t+\tau))-F(t+\tau)}{\mu(t+\tau)}
=f(t+\tau).
\]

$\mathbf{Case}$ (2):  If $t$ is right-dense, since $F$ is differentiable at $t$, $\lim\limits_{s\rightarrow t}\frac{F(t)- F(s)}{t-s}$ exists and is a finite number.
In this case $f(t)=F^{\Delta}(t)=\lim\limits_{s\rightarrow t}\frac{F(t)-F(s)}{t-s}$, so
\[
f(t+\tau)=\lim_{s^{'}\rightarrow t+\tau}\frac{F(t+\tau)-F(s^{'})}{t+\tau-s^{'}}
=\lim_{s\rightarrow t}\frac{F(t+\tau)-F(s+\tau)}{t+\tau-(s+\tau)}=\lim_{s\rightarrow t}\frac{G(t)-G(s)}{t-s}=G^{\Delta}(t).
\]
Therefore, for both cases we have $G^{\Delta}(t)=f(t+\tau)$ for all $t\in\mathbb{T}$. Consequently,
\[
\int_{a+\tau}^{b+\tau}f(t)\Delta t=F(b+\tau)-F(a+\tau)=G(b)-G(a)=\int_a^{b}f(t+\tau)\Delta t.
\]
The proof is complete.
\end{proof}

\begin{example}
Consider the time scale $\mathbb{T}=\bigcup\limits_{k=-\infty}^{+\infty}[2k,2k+1]$. Let $f$ be a right-dense continuous function. Obviously for this time scale, $2\in\Pi$ and
\[
\int_0^{2}f(t+2)\Delta t=\int_0^{1}f(t+2)dt+\int_1^{\sigma(1)}f(t+2)\Delta t=\int_2^{3}f(t)dt+\mu(1)f(3)=\int_2^{3}f(t)dt+f(3),
\]
\[
\int_2^{4}f(t)\Delta t=\int_2^{3}f(t)dt+\int_3^{\sigma(3)}f(t)\Delta t=\int_2^{3}f(t)dt+\mu(3)f(3)=\int_2^{3}f(t)dt+f(3).
\]
\end{example}
Now, for convenience, we introduce some notations. We will use
$x=(x_1,x_2,\ldots,x_n)^T\in\mathbb{R}^n$ to denote a column
vector, in which the symbol $T$ denotes the transpose of vectors. We
let $|x|$ denote the absolute-value vector given by
$|x|=(|x_1|,|x_2|,\ldots,|x_n|)^T$, and define
$||x||=\max\limits_{1\leq i\leq n}|x_i|$.

Let $BC(\mathbb{T},\mathbb{R}^n)=\{f:\mathbb{T}\rightarrow\mathbb{R}^n|f$ is bounded continuous function on $\mathbb{T}\}$  with the sup-norm defined by $||f||_{\infty}=\sup\limits_{t\in\mathbb{T}}\|f(t)\|$.
 It is easy to check that $(BC(\mathbb{T},\mathbb{R}^n),||\cdot||_{\infty})$ is a Banach space.

\begin{definition} \cite{19,20}
 Let $\mathbb{T}$ be an almost periodic time scale. A function $f\in C(\mathbb{T},\mathbb{R}^n)$ is called almost periodic if for each
 $\varepsilon>0$, there exists $l_{\varepsilon}>0$ such that every interval of length $l_{\varepsilon}$ contains at least a number
 $\tau\in\Pi$ with the following property
 \[
 \|f(t+\tau)-f(t)\|<\varepsilon,\,\,\,\,\,\,\forall t\in\mathbb{T}.
 \]
\end{definition}
The collection of all almost periodic functions which go from $\mathbb{T}$ to $\mathbb{R}^n$ will
be denoted by $AP(\mathbb{T},\mathbb{R}^n)$. $AP(\mathbb{T},\mathbb{R}^n)$ equipped with the sup-norm is a
Banach space.

\begin{definition} \cite{23}\label{def25}
A function $f\in C(\mathbb{T},\mathbb{R}^n)$ is called pseudo-almost periodic if $f=g+h$, where $g\in AP(\mathbb{T},\mathbb{R}^n)$ and
$h\in PAP_0^*(\mathbb{T},\mathbb{R}^n):=\big\{\varphi\in BC(\mathbb{T},\mathbb{R}^n):\varphi$ is $\Delta$-measurable such that
$\lim\limits_{r\rightarrow +\infty}\frac{1}{2r}\int_{\overline{t}-r}^{\overline{t}+r}|\varphi(s)|\Delta s=0$\big\},
where $\overline{t}\in\mathbb{T},\,r\in\Pi$.
\end{definition}

\begin{definition}
The weighted pseudo-almost periodic solution
$x^*(t)=(x^*_1(t),x^*_2(t),\ldots,x^*_n(t))^T$
of system \eqref{e1} with the initial value
$\varphi^*(t)=(\varphi^*_1(t),\varphi^*_2(t),\ldots,\varphi^*_n(t))^T$ is said to be globally exponentially stable.
If there exist  a positive constant $\lambda$ with
$\ominus\lambda\in \mathcal{R}^+$ and $M>1$ such that every solution
$x(t)=(x_1(t),x_2(t),\ldots,x_n(t))^T$ of
system \eqref{e1} with the initial value
$\varphi(t)=(\varphi_1(t),\varphi_2(t),\ldots,\varphi_n(t))^T$ satisfies
\[
\|x(t)-x^*(t)\|\leq M
e_{\ominus\lambda}(t,t_0)\|\psi\|_{\infty},\,\,\,\,\forall
t\in(0,+\infty)_{\mathbb{T}},
\]
where
\[
\|\psi\|_{\infty}=\sup_{t\in[-\gamma,0]_{\mathbb{T}}}\max_{1\leq
i\leq
n}|\varphi_i(t)-\varphi_i^*(t)|,\,\,\,t_0=\max\{[-v,0]_{\mathbb{T}}\}.
\]
\end{definition}

\section{Weighted pseudo-almost periodic functions on time scales}\setcounter{equation}{0}
\vspace{1ex}
\indent

Let $\mathbb{U}$ denote the collection of functions (weights) $u:\mathbb{T}\rightarrow(0,+\infty)$, which are locally
integrable over $\mathbb{T}$ such that $u>0$ almost everywhere. Let $u\in\mathbb{U}$,   for $r\in\Pi$ with $r>0$,
we  denote
\[
u(Q_r):=\int_{Q_r}u(x)\Delta x,
\]
where $Q_r:=[\bar{t}-r,\bar{t}+r]_{\mathbb{T}}\,\,(\bar{t}=\min\{[0,\infty)_\mathbb{T}\})$.
If $u(x)=1$ for each $x\in\mathbb{T}$, then $\lim\limits_{r\rightarrow\infty}u(Q_r)=\infty$. Consequently, we define the space of weights $\mathbb{U}_{\infty}$ by $\mathbb{U}_{\infty}:=\big\{u\in\mathbb{U}:\inf\limits_{t\in\mathbb{T}}u(t)=u_0>0,\,\lim\limits_{r\rightarrow\infty}
u(Q_r)=\infty\big\}$. In addition  to the above, we define the set of weights $\mathbb{U}_{B}$ by $\mathbb{U}_{B}:=\big\{u\in\mathbb{U}_{\infty}:\sup\limits_{t\in\mathbb{T}}u(t)<\infty\big\}$.
\begin{definition}\label{def31}
Fix $u\in\mathbb{U}_{\infty}$. A continuous function $f:\mathbb{T}\rightarrow\mathbb{R}^n$ is called weighted pseudo-almost periodic if it can be
written as $f=h+\varphi$ with $h\in AP(\mathbb{T},\mathbb{R}^n)$ and $\varphi\in PAP_0(\mathbb{T},\mathbb{R}^n,u)$, where the space
$PAP_0(\mathbb{T},\mathbb{R}^n,u)$ is defined by
\[
PAP_0(\mathbb{T},\mathbb{R}^n,u)=\bigg\{\varphi\in BC(\mathbb{T},\mathbb{R}^n):\lim_{r\rightarrow +\infty}
\frac{1}{u(Q_r)}\int_{Q_r}\|g(t)\|u(t)\Delta t=0\bigg\}.
\]
$PAP_0(\mathbb{T},\mathbb{R}^n,u)$ is called translation invariant if $\varphi \in  PAP_0(\mathbb{T},\mathbb{R}^n,u)$, then $\varphi_\tau\in PAP_0(\mathbb{T},\mathbb{R}^n,u)$, where $\tau\in \Pi$ and $\varphi_\tau (t)= \varphi(t-\tau)$.
\end{definition}
All weighted pseudo-almost periodic functions which go from $\mathbb{T}$ to $\mathbb{R}^n$, will be denoted by $PAP(\mathbb{T},\mathbb{R}^n,u)$.
By Definition 2.5, it is easy to see that the decomposition of a pseudo-almost periodic function on time scales is unique. This is not always the case for weighted pseudo-almost
periodic functions (see Example \ref{ex31}). In fact, the uniqueness of such a decomposition depends upon the translation-invariance of the space
$PAP_0(\mathbb{T},\mathbb{R}^n,u)$.
\begin{example}\label{ex31}
Consider the almost periodic time scale $\mathbb{T}=\bigcup\limits_{k=-\infty}^{+\infty}[2k,2k+1]$. Let $u(t)=e^{|t|}$.
Obviously $\inf\limits_{t\in\mathbb{T}}u(t)=1>0,\,\Pi=2\mathbb{Z}$, and
{\setlength\arraycolsep{2pt}
\begin{eqnarray*}
 \lim_{r\rightarrow\infty}u(Q_r)
&=&\lim_{k\rightarrow\infty}\int_{-2k}^{2k}e^{|t|}\Delta t\\
&=&\lim_{k\rightarrow\infty}\bigg[\int_{-2k}^{-2k+1}e^{-t}dt+\int_{-2k+1}^{\sigma(-2k+1)}e^{-t}\Delta t
+\int_{-2k+2}^{-2k+3}e^{-t}dt\\
&&+\int_{-2k+3}^{\sigma(-2k+3)}e^{-t}\Delta t+\ldots +\int_{-2}^{-1}e^{-t}dt+\int_{-1}^{\sigma(-1)}e^{-t}\Delta t
+\int_0^{1}e^{t}dt\\
&&+\int_1^{\sigma(1)}e^{t}\Delta t+\int_2^{3}e^{t}dt+\int_3^{\sigma(3)}e^{t}\Delta t+\ldots+\int_{2k-2}^{2k-1}e^{t}dt+\int_{2k-1}^{\sigma(2k-1)}e^{t}\Delta t\bigg]\\
&=&\lim_{k\rightarrow\infty}[e^{2k}-1+2(e^{1}+e^{3}+\ldots+e^{2k-1})]=\infty,
\end{eqnarray*}}
which implies that $u\in\mathbb{U}_{\infty}$. It is easy to see that $f(t)=(\sin(2\pi t),\sin(4\pi t))^T\in AP(\mathbb{T},\mathbb{R}^2)$. For each $i\in\mathbb{N}$, we have
\begin{eqnarray*}
\int_{-2i}^{-2i+1}|\sin(2\pi t)|e^{|t|}\Delta t&=&\int_{-2i}^{-2i+1}|\sin(2\pi t)|e^{-t}dt=0,
\end{eqnarray*}
\[
\int_{-2i+1}^{\sigma(-2i+1)}|\sin(2\pi t)|e^{|t|}\Delta t=\mu(-2i+1)|\sin(2\pi(-2i+1))|e^{|-2i+1|}=0,
\]
\begin{eqnarray*}
\int_{2i-2}^{2i-1}|\sin(2\pi t)|e^{|t|}\Delta t&=&\int_{2i-2}^{2i-1}|\sin(2\pi t)|e^{t}dt=0,
\end{eqnarray*}
\[
\int_{2i-1}^{\sigma(2i-1)}|\sin(2\pi t)|e^{|t|}\Delta t=\mu(2i-1)|\sin(2\pi-2i+1))|e^{|2i-1|}=0,
\]
so $\frac{1}{u(Q_r)}\int_{Q_r}|\sin(2\pi t)|e^{|t|}\Delta t=0$. Similarly, $\frac{1}{u(Q_r)}\int_{Q_r}|\sin(4\pi t)|e^{|t|}\Delta t=0$,
that is, $f(t)\in AP(\mathbb{T},\mathbb{R}^2)\cap PAP_0(\mathbb{T},\mathbb{R}^2,u)$.
\end{example}
Similar to the proof of Theorem 2.1 in \cite{24}, we have
\begin{theorem}\label{thm31}
Suppose that $u\in\mathbb{U}_{\infty}$  and for any $\tau\in\Pi$, $\overline{\lim\limits_{|t|\rightarrow +\infty}}\frac{u(t+\tau)}{u(t)}$ is finite, then
$\overline{\lim\limits_{|t|\rightarrow +\infty}}\frac{u(Q_{t+\tau})}{u(Q_t)}$ is finite and $PAP_0(\mathbb{T},\mathbb{R}^n,u)$ is translation invariant.
\end{theorem}
Based on Theorem \ref{thm31}, we introduce the following new set of weights, which makes the spaces of weighted pseudo-almost periodic functions translation invariant:
$$\mathbb{U}_{\infty}^{Inv}:=\bigg\{u\in\mathbb{U}_{\infty}: \,\,\mathrm{for}\,\, \mathrm{all}\,\, s\in\Pi,
 \overline{\lim\limits_{|t|\rightarrow\infty}}\frac{u(t+s)}{u(t)}<\infty\bigg\}.$$
By Definition \ref{def31}  and the definition of $\mathbb{U}_{\infty}^{Inv}$, one can easily show that
\begin{lemma}\label{lem31}
 Let $u\in\mathbb{U}_{\infty}^{Inv}$. If $f,g\in PAP(\mathbb{T},\mathbb{R}^n,u)$, then $f+g,fg\in PAP(\mathbb{T},\mathbb{R}^n,u)$; if $f\in PAP(\mathbb{T},\mathbb{R}^n,u),\,
g\in AP(\mathbb{T},\mathbb{R}^n)$, then $fg\in PAP(\mathbb{T},\mathbb{R}^n,u)$.
\end{lemma}
\begin{theorem}\label{thm32}
Let $u\in\mathbb{U}_{\infty}^{Inv}$. If $f\in PAP(\mathbb{T},\mathbb{R}^n,u)$, then there exist a unique $g\in AP(\mathbb{T},\mathbb{R}^n)$
and a unique $h\in PAP_0(\mathbb{T},\mathbb{R}^n,u)$ such that $f=g+h$.
\end{theorem}
\begin{proof}
Suppose that $f\not\equiv0$ and $f\in AP(\mathbb{T},\mathbb{R}^n)\cap PAP_0(\mathbb{T},\mathbb{R}^n,u)$. Then there exists a $t_0\in\mathbb{T}$ such
that $f(t_0)\neq 0$. One can assume that there exists $\delta>0$ such that $\|f(t_0)\|\geq 2\delta$. Define
\[
B_{\delta}:=\{\tau\in\Pi:\|f(t_0+\tau)-f(t_0)\|\leq\delta\}
\]
for every $ r>0$ and $r\in\Pi$. For any $t\in[\overline{t}-r,\overline{t}+r]_{\mathbb{T}}\, \,(\overline{t}=\min\{[0,+\infty)_{\mathbb{T}}\})$,
since $f\in AP(\mathbb{T},\mathbb{R}^n)$,  there exists $l_{\delta}>0$ and
 $\tau\in[t-l_{\delta},t]_{\mathbb{T}}\cap B_{\delta}$, we have
 \[
 t=\tau+(t-\tau)\in (t-\tau)+B_{\delta}\subset\bigcup_{s\in\mathbb{T}}(s+B_{\delta}).
 \]
 On the other hand, $[\overline{t}-r,\overline{t}+r]_{\mathbb{T}}$ is a bounded closed subset of $\mathbb{R}$, so
 $[\overline{t}-r,\overline{t}+r]_{\mathbb{T}}$ is a compact subset of $\mathbb{R}$, then there exist $s_1,s_2,\ldots,s_m\in\mathbb{T}$ such that
 \[
 [\overline{t}-r,\overline{t}+r]_{\mathbb{T}}\subset\bigcup_{k=1}^{m}(s_k+B_{\delta}).
 \]
 Noticing that
 \[
 \|f(t_{0}+t)\|\geq\|f(t_0)\|-\|f(t_{0}+t)-f(t_0)\|\geq\delta
 \]
 for all $t\in B_{\delta}$. For every $ t\in[\overline{t}-r,\overline{t}+r]_{\mathbb{T}}$, there exists an $i\in\{1,2,\ldots,m\}$
 such that $t-s_i\in B_{\delta}$, hence $\|f(t-s_{i}+t_0)\|\geq\delta$. Set
 \[
 F(t)=\|f(t+t_0)\|+\|f(t+t_{0}-s_1)\|+\|f(t+t_{0}-s_2)\|+\ldots+\|f(t+t_{0}-s_m)\|.
 \]
 One can easily see that $F(t)\geq\delta$ for all $t\in[\overline{t}-r,\overline{t}+r]_{\mathbb{T}}$. So
 \begin{equation} \label{c1}
 \frac{1}{u(Q_r)}\int_{Q_r}F(t)u(t)\Delta t\geq\frac{\delta}{u(Q_r)}\int_{Q_r}u(t)\Delta t=\delta.
 \end{equation}
 Using the fact that $PAP_0(\mathbb{T},\mathbb{R}^n,u)$ is translation-invariant and $f\in PAP_0(\mathbb{T},\mathbb{R}^n,u)$, it follows that
 $f(t+t_0),f(t+t_{0}-s_k)\,(k=1,2,\ldots,m)\in PAP_0(\mathbb{T},\mathbb{R}^n,u)$, that is,
 \[
 \lim_{r\rightarrow\infty}\frac{1}{u(Q_r)}\int_{Q_r}\|f(t+t_0)\|u(t)\Delta t=0
 \]
 and
 \[
 \lim_{r\rightarrow\infty}\frac{1}{u(Q_r)}\int_{Q_r}\|f(t+t_{0}-s_k)\|u(t)\Delta t=0,\,\,\,k=1,2,\ldots,m,
 \]
 which contradict  \eqref{c1}, and hence $AP(\mathbb{T},\mathbb{R}^n)\cap PAP_0(\mathbb{T},\mathbb{R}^n,u)=\{0\}$,
 that is, $PAP(\mathbb{T},\mathbb{R}^n,u)=AP(\mathbb{T},\mathbb{R}^n)\bigoplus PAP_0(\mathbb{T},\mathbb{R}^n,u)$. The proof is complete.
\end{proof}

\begin{lemma}\label{lem32}
Let $u\in\mathbb{U}_{\infty}^{Inv}$. If $f=g+h\in PAP(\mathbb{T},\mathbb{R}^n,u)$, where $g\in AP(\mathbb{T},\mathbb{R}^n)$, then $g(\mathbb{T})\subset
\overline{f(\mathbb{T})}$ and $||g||_{\infty}\leq||f||_{\infty}$.
\end{lemma}
\begin{proof}
If we suppose that $g(\mathbb{T})\subset\overline{f(\mathbb{T})}$ does not hold, then exist $t_0\in\mathbb{T}$ and
$\varepsilon_0>0$ such that
$
\inf_{s\in\mathbb{T}}\|g(t_0)-f(s)\|>\varepsilon_0.
$
Using the continuity of the function $g$,  there exists $\delta>0$ such that for $t\in(t_{0}-\delta,t_{0}+\delta)\cap\mathbb{T}$,
$
\|g(t)-g(t_0)\|<\frac{\varepsilon_0}{2}.
$
Since $g\in AP(\mathbb{T},\mathbb{R}^n)$,   there exists $l_{\frac{\varepsilon_0}{4}}>0$
such that every interval of length $l_{\frac{\varepsilon_0}{4}}$ contains a $\tau\in \Pi$ with the property that
\[
\|g(t+\tau)-g(t)\|<\frac{\varepsilon_0}{4},\,\, t\in \mathbb{T}.
\]
So
{\setlength\arraycolsep{2pt}
\begin{eqnarray*}
\|h(t+\tau)\|&=&\|f(t+\tau)-g(t+\tau)\|\\
&\geq&\|f(t+\tau)-g(t)\|-\|g(t)-g(t+\tau)\|\\
&\geq&\|f(t+\tau)-g(t_0)\|-\|g(t_0)-g(t)\|-\|g(t)-g(t+\tau)\|\\
&>&\frac{\varepsilon_0}{4},\,\,\forall t\in(t_{0}-\delta,t_{0}+\delta)\cap\mathbb{T},
\end{eqnarray*}}
which implies that
\[
 \lim_{r\rightarrow\infty}\frac{1}{u(Q_r)}\int_{Q_r}\|h(t+\tau)\|u(t)\Delta t\geq\frac{ \varepsilon_0}{4}
 \]
 and
this is a contradiction. The proof is complete.
\end{proof}

\begin{lemma}\label{lem33}
 Let $u\in\mathbb{U}_{\infty}^{Inv}$. If $(f_m)_{m\in\mathbb{N}}\subset APA(\mathbb{T},\mathbb{R}^n,u)$ such that $\lim\limits_{m\rightarrow +\infty}||f_{m}-f||_{\infty}=0$,
then $f\in PAP(\mathbb{T},\mathbb{R}^n,u)$.
\end{lemma}
\begin{proof}
Since $(f_m)_{m\in\mathbb{N}}\subset PAP(\mathbb{T},\mathbb{R}^n,u)$, there exist $(g_m)_{m\in\mathbb{N}}\subset AP(\mathbb{T},\mathbb{R}^n)$
and $(h_m)_{m\in\mathbb{N}}\subset PAP_0(\mathbb{T},\mathbb{R}^n,u)$ such that $f_m=g_m+h_m,\,\,\forall m\in\mathbb{N}$. According to Lemma \ref{lem31} and Theorem \ref{thm32},
we have $f_s-f_m=(g_s-g_m)+(h_s-h_m)\in PAP(\mathbb{T},\mathbb{R}^n,u), g_s-g_m\in AP(\mathbb{T},\mathbb{R}^n)$ and $h_s-h_m\in PAP_0(\mathbb{T},\mathbb{R}^n,u)$. From Lemma \ref{lem32} it follows that $||g_{s}-g_m||_{\infty}\leq||f_{s}-f_m||_{\infty}$.
Since $\lim\limits_{m\rightarrow \infty}||f_{m}-f||_{\infty}=0$ it follows that $(f_m)_{m\in\mathbb{N}}$ is a cauchy sequence, and
$||g_{s}-g_m||_{\infty}\rightarrow 0$ as $s,m\rightarrow +\infty$ too, that is, $(g_m)_{m\in\mathbb{N}}$ is a cauchy sequence. Using the fact that
$(AP(\mathbb{T},\mathbb{R}^n),||\cdot||_{\infty})$ is a Banach space, it follows that there exists $g\in AP(\mathbb{T},\mathbb{R}^n)$ such that $\lim\limits_{m\rightarrow \infty}||g_{m}-g||_{\infty}=0$.
Let $h=f-g$, it is easy to see that $\lim\limits_{m\rightarrow \infty}||h_{m}-h||_{\infty}=0$, which yields that $h\in BC(\mathbb{T},\mathbb{R}^n)$.
On the other hand,
{\setlength\arraycolsep{2pt}
\begin{eqnarray*}
\frac{1}{u(Q_r)}\int_{Q_r}||h(t)||u(t)\Delta t&=&\frac{1}{u(Q_r)}\int_{Q_r}||h(t)-h_m(t)+h_m(t)||u(t)\Delta t\\
&\leq&\frac{1}{u(Q_r)}\int_{Q_r}||h(t)-h_m(t)||u(t)\Delta t+\frac{1}{u(Q_r)}\int_{Q-r}||h_m(t)||u(t)\Delta t\\
&\leq&||h_{m}-h||_{\infty}+\frac{1}{u(Q_r)}\int_{Q_r}||h_m(t)||u(t)\Delta t.
\end{eqnarray*}}
Letting $r\rightarrow \infty$, it follows that
\[
\lim_{r\rightarrow \infty}\frac{1}{u(Q_r)}\int_{Q_r}||h(t)||u(t)\Delta t\leq||h_{m}-h||_{\infty}.
\]
Letting $m\rightarrow \infty$ in the previous inequality, we get
\[
\lim_{r\rightarrow \infty}\frac{1}{u(Q_r)}\int_{Q_r}||h(t)||u(t)\Delta t=0,
\]
that is, $h\in PAP_0(\mathbb{T},\mathbb{R}^n,u)$. The proof is complete.
\end{proof}
\begin{corollary}
 Let $u\in\mathbb{U}_{\infty}^{Inv}$. Then
$(PAP(\mathbb{T},\mathbb{R}^n,u),||\cdot||_{\infty})$ is a Banach space.
\end{corollary}
\begin{proof}
 By Lemma \ref{lem33}, $PAP(\mathbb{T},\mathbb{R}^n,u)$ is closed and $PAP(\mathbb{T},\mathbb{R}^n,u)\subset BC(\mathbb{T},\mathbb{R}^n)$. Therefore,
$(PAP(\mathbb{T},\mathbb{R}^n,u),||\cdot||_{\infty})$ is a Banach space. The proof is complete.
\end{proof}
\begin{lemma}\label{le51}
 Let $u\in\mathbb{U}_{\infty}^{Inv}$. If $f:\mathbb{R}\rightarrow\mathbb{R}$ satisfies the Lipschitz condition and $\varphi\in PAP(\mathbb{T},\mathbb{R},u), \tau\in\Pi$,
then $\Gamma:t\rightarrow f(\varphi(t-\tau))$ belongs to $PAP(\mathbb{T},\mathbb{R},u)$.
\end{lemma}
\begin{proof}
Since $\varphi\in PAP(\mathbb{T},\mathbb{R},u)$, there exist $\varphi_1\in AP(\mathbb{T},\mathbb{R})$ and $\varphi_2\in PAP_0(\mathbb{T},\mathbb{R},u)$
such that $\varphi=\varphi_1+\varphi_2$. Set
\[
\Gamma(t)=f(\varphi(t-\tau))=f(\varphi_1(t-\tau))+[f(\varphi_1(t-\tau)+\varphi_2(t-\tau))-f(\varphi_1(t-\tau))]:=\Gamma_1(t)+\Gamma_2(t).
\]
First, we prove that $\Gamma_1\in AP(\mathbb{T},\mathbb{R})$. Since $f$ satisfies the Lipschitz condition, there
exists a positive constant $L$ such that $|f(u_1)-f(u_2)|\leq L|u_1-u_2|$, $\forall u_1,u_2\in\mathbb{R}$.
For any $\varepsilon>0$, since $\varphi_1\in AP(\mathbb{T},\mathbb{R})$, it is possible to find a real number
$l=l(\varepsilon)>0$, for any interval with length $l(\varepsilon)$, there exists a number $\alpha=\alpha(\varepsilon)\in \Pi$
in this interval such that $|\varphi_1(t+\alpha)-\varphi_1(t)|<\frac{\varepsilon}{L}$ for $\forall t\in\mathbb{T}$, then
\[
|\Gamma_1(t+\alpha)-\Gamma_1(t)|=|f(\varphi_1(t+\alpha-\tau))-f(\varphi_1(t-\tau))|\leq L|\varphi_1(t+\alpha-\tau)-\varphi_1(t-\tau)|<\varepsilon,
\]
which implies that $\Gamma_1\in AP(\mathbb{T},\mathbb{R})$. Next we prove that $\Gamma_2\in PAP_0(\mathbb{T},\mathbb{R},u)$. Since $\varphi_2\in PAP_0(\mathbb{T},\mathbb{R},u)$, by using Theorem \ref{thm31}, we have $\varphi_2(t-\tau)\in PAP_0(\mathbb{T},\mathbb{R},u)$, so
{\setlength\arraycolsep{2pt}
\begin{eqnarray*}
\lim_{r\rightarrow\infty}\frac{1}{u(Q_r)}\int_{Q_r}|\Gamma_2(t)|u(t)\Delta t&=&\lim_{r\rightarrow\infty}\frac{1}{u(Q_r)}\int_{Q_r}|f(\varphi_1(t-\tau)+\varphi_2(t-\tau))-f(\varphi_1(t-\tau))|u(t)\Delta t\\
&\leq&\lim_{r\rightarrow\infty}\frac{L}{u(Q_r)}\int_{Q_r}|\varphi_2(t-\tau)|u(t)\Delta t=0,
\end{eqnarray*}}
which implies that $\Gamma_2\in PAP_0(\mathbb{T},\mathbb{R},u)$. Consequently, $\Gamma\in PAP(\mathbb{T},\mathbb{R},u)$. The proof is complete.
\end{proof}

\section{Weighted pseudo-almost periodic solutions of linear dynamic equations on time scales}\setcounter{equation}{0}
\vspace{1ex}
\indent

Consider the non-autonomous equation
\begin{equation} \label{d1}
x^{\Delta}=A(t)x+F(t)
\end{equation}
and its associated homogeneous equation
\begin{equation} \label{d2}
x^{\Delta}=A(t)x,
\end{equation}
where the $n\times n$ coefficient matrix $A(t)$ is continuous on $\mathbb{T}$ and the column vector $F=(f_1,f_2,\ldots,f_n)^T: \mathbb{T}\rightarrow \mathbb{R}^n$. Define $\|F\|=\sup\limits_{t\in\mathbb{T}}\|F(t)\|$. We will call $A(t)$ is almost periodic if all of its entries are almost periodic.
\begin{definition} \cite{19}\label{def41}
Equation \eqref{d2} is said to admit an exponential dichotomy on $\mathbb{T}$ if there
exist positive constants $k,\alpha$, projection $P$ and the
fundamental solution matrix $X(t)$ of \eqref{d2}, satisfying
\[
\|X(t)PX^{-1}(\sigma(s))\|_0\leq
ke_{\ominus\alpha}(t,\sigma(s)),\,\,\,s,t\in\mathbb{T},\,\,t\geq\sigma(s),
\]
\[
\|X(t)(I-P)X^{-1}(\sigma(s))\|_0\leq
ke_{\ominus\alpha}(\sigma(s),t),\,\,\,s,t\in\mathbb{T},\,\,t\leq\sigma(s),
\]
where $\|\cdot\|_0$ is a matrix norm on $\mathbb{T}$.
\end{definition}
\begin{lemma}\label{lem41}
Suppose $a>0$, then
\[
e_{\ominus a}(t,s)\leq\exp\bigg( \frac{-a}{1+\overline{\mu}a}(t-s)\bigg),\,\,\forall s\leq t,
\]
where $\overline{\mu}
=\sup\limits_{t\in\mathbb{T}}\mu(t)$.
\end{lemma}
\begin{proof}For every $\tau\in \mathbb{T}$, if $\mu(\tau)= 0$, then
\[
\xi_{\mu(\tau)}(\ominus a)=\frac{-a}{1+\mu(\tau)a}=-a\leq\frac{-a}{1+1+\overline{\mu}a};
\]
if $\mu(\tau)>0$, then
{\setlength\arraycolsep{2pt}
\begin{eqnarray*}
\xi_{\mu(\tau)}(\ominus a)&=&\frac{\log(1+\mu(\tau)\ominus a)}{\mu(\tau)}=\frac{\log(1-\mu(\tau)\frac{a}{1+\mu(\tau)a})}{\mu(\tau)}=\frac{-\log(1+\mu(\tau)a)}{\mu(\tau)}\\
&\leq&\frac{\frac{-\mu(\tau)a}{1+\mu(\tau)a}}{\mu(\tau)}=\frac{-a}{1+\mu(\tau)a}\leq\frac{-a}{1+\overline{\mu}a}.
\end{eqnarray*}}
Thus, we have
\[
\xi_{\mu(\tau)}(\ominus a)\leq\frac{-a}{1+\overline{\mu}a},\,\,\,\,\forall \tau\in\mathbb{T},
\]
so
\[
e_{\ominus a}(t,s)=\exp\bigg(\int_s^{t}\xi_{\mu(\tau)}(\ominus a)\Delta\tau\bigg)\leq\exp\bigg(\int_s^{t}\frac{-a}{1+\overline{\mu}a}\Delta\tau\bigg)
=\exp\bigg(\frac{-a}{1+\overline{\mu}a}(t-s)\bigg).
\]
The proof is complete.
\end{proof}
\begin{lemma}\cite{20}\label{lem42}
Let $c_{i}(t)$ be an almost periodic function on $\mathbb{T}$, where
$c_{i}(t)>0, -c_{i}(t)\in \mathcal{R}^{+}, i=1,2,\ldots,n, \forall t\in
\mathbb{T}$ and $\min\limits_{1 \leq i \leq n}\{\inf\limits_{t\in
\mathbb{T}}c_i(t)\}=\widetilde{m}
> 0$, then the linear system
\begin{eqnarray*}\label{e23}
x^{\Delta}(t)=\mathrm{diag}(-c_{1}(t),-c_{2}(t),\dots,-c_{n}(t))x(t)
\end{eqnarray*}
admits an exponential dichotomy on $\mathbb{T}$.
\end{lemma}

\begin{theorem}\label{thm41}
Let $u\in\mathbb{U}_{\infty}^{Inv}$. Assume that $A(t)$ is almost periodic, \eqref{d2} admits an exponential dichotomy and the function $F\in PAP_0(\mathbb{T},\mathbb{R}^n,u)$.
Then \eqref{d1} has a unique bounded solution $x\in PAP_0(\mathbb{T},\mathbb{R}^n,u)$.
\end{theorem}
\begin{proof}
Similar to the proof of Theorem 5.1 in \cite{23}, we have
\[
x(t)=\int^t_{-\infty}X(t)PX^{-1}(\sigma(s))F(s)\Delta s-\int_t^{+\infty}X(t)(I-P)X^{-1}(\sigma(s))F(s)\Delta s
\]
is a unique bounded solution of \eqref{d1}. Next, we will show that $x\in PAP_0(\mathbb{T},\mathbb{R}^n,u)$. Let
\[
I(t)=\int^t_{-\infty}X(t)PX^{-1}(\sigma(s))F(s)\Delta s,\,H(t)=\int_t^{+\infty}X(t)(I-P)X^{-1}(\sigma(s))F(s)\Delta s.
\]
By using Lemma \ref{lem23}, Lemma \ref{lem41}  and in view of Definition \ref{def41}, we can get
{\setlength\arraycolsep{2pt}
\begin{eqnarray*}
\lim_{r\rightarrow \infty}\frac{1}{u(Q_r)}\int_{Q_r}\|I(t)\|u(t)\Delta t
&=&\lim_{r\rightarrow \infty}\frac{1}{u(Q_r)}\int_{Q_r}\big\|\int^t_{-\infty}X(t)PX^{-1}(\sigma(s))F(s)\Delta s\big\|u(t)\Delta t\\
&\leq&\lim_{r\rightarrow \infty}\frac{1}{u(Q_r)}\int_{Q_r}\big(\int^t_{-\infty}\|X(t)PX^{-1}(\sigma(s))\|\|F(s)\|\Delta s\big)u(t)\Delta t\\
&\leq&\lim_{r\rightarrow \infty}\frac{1}{u(Q_r)}\int_{Q_r}\big(\int^t_{-\infty}K e_{\ominus\alpha}(t,\sigma(s))\|F(s)\|\Delta s\big)u(t)\Delta t\\
&\leq&\lim_{r\rightarrow \infty}\frac{1}{u(Q_r)}\int_{Q_r}\big(\int^t_{-\infty}K e^{-\frac{\alpha}
{1+\overline{u}\alpha}(t-\sigma(s))}\|F(s)\|\Delta s\big)u(t)\Delta t\\
&\leq&\lim_{r\rightarrow \infty}\frac{1}{u(Q_r)}\int_{Q_r}\big(\int^t_{-\infty}K e^{-\frac{\alpha}
{1+\overline{u}\alpha}(t-s-k)}\|F(s)\|\Delta s\big)u(t)\Delta t\\
&=&\lim_{r\rightarrow \infty}\frac{1}{u(Q_r)}\int_{Q_r}\big(\int^{+\infty}_{-k}K e^{-\frac{\alpha}
{1+\overline{u}\alpha}s}\|F(t-s-k)\|\Delta s\big)u(t)\Delta t\\
&=&\lim_{r\rightarrow \infty}\int_{-k}^{+\infty}K e^{-\frac{\alpha}
{1+\overline{u}\alpha}s}\big(\frac{1}{u(Q_r)}\int_{Q_r}\|F(t-s-k)\|u(t)\Delta t\big)\Delta s,
\end{eqnarray*}}
where $k(>0)\in \Pi$. Consider the following function
\[
\Gamma_r(s)=\frac{1}{u(Q_r)}\int_{Q_r}\|F(t-s-k)\|\mu(t)\Delta t.
\]
Obviously, $\Gamma_r(s)$ is bounded. By using Corollary \ref{cl22}, we have that $\Gamma_r(s)$ is $\Delta$-measurable and by using Theorem \ref{thm31}, we can get
$\lim\limits_{r\rightarrow\infty}\Gamma_r(s)=0$. Consequently, by Theorem \ref{thm21}, we obtain
{\setlength\arraycolsep{2pt}
\begin{eqnarray} \label{d3}
\lim_{r\rightarrow\infty}\frac{1}{u(Q_r)}\int_{Q_r}\|I(t)\|u(t)\Delta t
&=&\lim_{r\rightarrow\infty}\int_{-k}^{+\infty}K e^{-\frac{\alpha}{1+\overline{u}\alpha}s}\Gamma_r(s)\Delta s\nonumber\\
&=&\int_{-k}^{+\infty}\lim_{r\rightarrow\infty}(K e^{-\frac{\alpha}{1+\overline{u}\alpha}s}\Gamma_r(s))\Delta s=0.
\end{eqnarray}}
By using Lemma \ref{lem23}, Lemma \ref{lem41}  and in view of Definition \ref{def41}, we can obtain
{\setlength\arraycolsep{2pt}
\begin{eqnarray*}
&&\lim_{r\rightarrow\infty}\frac{1}{u(Q_r)}\int_{Q_r}\|H(t)\|u(t)\Delta t\\
&=&\lim_{r\rightarrow\infty}\frac{1}{u(Q_r)}
\int_{Q_r}\big\|\int_t^{+\infty}X(t)(I-P)X^{-1}(\sigma(s))F(s)\Delta s\big\|u(t)\Delta t\\
&\leq&\lim_{r\rightarrow\infty}\frac{1}{u(Q_r)}\int_{Q_r}\big(\int_t^{+\infty}\|X(t)(I-P)X^{-1}(\sigma(s))\|\|F(s)\|\Delta s\big)u(t)\Delta t\\
&\leq&\lim_{r\rightarrow\infty}\frac{1}{u(Q_r)}\int_{Q_r}\big(K e_{\ominus\alpha}(\sigma(s),t)\|F(s)\|\Delta s\big)u(t)\Delta t\\
&\leq&\lim_{r\rightarrow\infty}\frac{1}{u(Q_r)}\int_{Q_r}\big(\int_t^{+\infty}K e^{-\frac{\alpha}{1+\overline{u}\alpha}
(\sigma(s)-t)}\|F(s)\|\Delta s\big)u(t)\Delta t\\
&\leq&\lim_{r\rightarrow\infty}\frac{1}{u(Q_r)}\int_{Q_r}\big(\int_t^{+\infty}K e^{-\frac{\alpha}{1+\overline{u}\alpha}
(s-t)}\|F(s)\|\Delta s\big)u(t)\Delta t\\
&=&\lim_{r\rightarrow\infty}\frac{1}{u(Q_r)}
\int_{Q_r}\big(\int_0^{+\infty}K e^{-\frac{\alpha}{1+\overline{u}\alpha}s}\|F(s+t)\|\Delta s\big)u(t)\Delta t\\
&=&\lim_{r\rightarrow\infty}\int_0^{+\infty}K e^{-\frac{\alpha}{1+\overline{u}\alpha}s}
\big(\frac{1}{u(Q_r)}\int_{Q_r}\|F(s+t\|u(t)\Delta t\big)\Delta s.
\end{eqnarray*}}
Let
\[
T_r(s)=\frac{1}{u(Q_r)}\int_{Q_r}\|F(s+t)\|u(t)\Delta t.
\]
It is easy to see that $T_r(s)$ is bounded. By using Corollary \ref{cl22}, we see that $T_r(s)$ is $\Delta$-measurable and by using
Theorem \ref{thm31}, we have $\lim\limits_{r\rightarrow\infty}T_r(s)=0$. Consequently, by Theorem \ref{thm21}, we get
{\setlength\arraycolsep{2pt}
\begin{eqnarray} \label{d4}
\lim_{r\rightarrow\infty}\frac{1}{u(Q_r)}\int_{Q_r}\|H(t)\|u(t)\Delta t&=&\lim_{r\rightarrow\infty}
\int_0^{+\infty}K e^{-\frac{\alpha}{1+\overline{u}\alpha}s}T_r(s)\Delta s\nonumber\\
&=&\int_0^{+\infty}\lim_{r\rightarrow\infty}(K e^{-\frac{\alpha}{1+\overline{u}\alpha}s}T_r(s))\Delta s=0.
\end{eqnarray}}
From \eqref{d3} and \eqref{d4}, we have
\[
\lim_{r\rightarrow\infty}\frac{1}{u(Q_r)}\int_{Q_r}\|x(t)\|u(t)\Delta t\leq\lim_{r\rightarrow\infty}
\frac{1}{u(Q_r)}\int_{Q_r}(\|I(t)\|+\|H(t)\|)u(t)\Delta t=0,
\]
which implies that $x(t)\in PAP_0(\mathbb{T},\mathbb{R}^n,u)$. The proof is complete.
\end{proof}
\begin{theorem}\label{thm42}
Let $u\in\mathbb{U}_{\infty}^{Inv}$. Suppose that $A(t)$ is almost periodic and (4.2) admits an exponential dichotomy. Then for every
$F\in PAP(\mathbb{T},\mathbb{R}^n,u)$, (4.1) has a unique bounded solution $x_{F}\in PAP(\mathbb{T},\mathbb{R}^n,u)$.
\end{theorem}
\begin{proof}
Since $F\in PAP(\mathbb{T},\mathbb{R}^n,u)$, $F=G+H$ where $G\in AP(\mathbb{T},\mathbb{R}^n)$ and $H\in PAP_0(\mathbb{T},\mathbb{R}^n,u)$.
According to the proof of Theorem \ref{thm41}, the function
{\setlength\arraycolsep{2pt}
\begin{eqnarray*}
x_{F}&=&\int^t_{-\infty}X(t)PX^{-1}(\sigma(s))F(s)\Delta s-\int_t^{+\infty}X(t)(I-P)X^{-1}(\sigma(s))F(s)\Delta s\\
&=&\bigg(\int^t_{-\infty}X(t)PX^{-1}(\sigma(s))G(s)\Delta s-\int_t^{+\infty}X(t)(I-P)X^{-1}(\sigma(s))G(s)\Delta s\bigg)\\
&&+\bigg(\int^t_{-\infty}X(t)PX^{-1}(\sigma(s))H(s)\Delta s-\int_t^{+\infty}X(t)(I-P)X^{-1}(\sigma(s))H(s)\Delta s\bigg)\\
&:=& x_{G}+x_{H}
\end{eqnarray*}}
is the unique solution of \eqref{d1}, where
\[
x_{G}:=\int^t_{-\infty}X(t)PX^{-1}(\sigma(s))G(s)\Delta s-\int_t^{+\infty}X(t)(I-P)X^{-1}(\sigma(s))G(s)\Delta s,
\]
\[
x_{H}:=\int^t_{-\infty}X(t)PX^{-1}(\sigma(s))H(s)\Delta s-\int_t^{+\infty}X(t)(I-P)X^{-1}(\sigma(s))H(s)\Delta s.
\]
By Theorem 4.1 in \cite{19}, $x_{G}\in AP(\mathbb{T},\mathbb{R}^n)$. By Theorem \ref{thm41}, $x_{H}\in PAP_0(\mathbb{T},\mathbb{R}^n,u)$.
Therefore, $x_{F}\in PAP(\mathbb{T},\mathbb{R}^n,u)$. This completes the proof.
\end{proof}

\section{Existence of weighted pseudo-almost periodic solutions of cellular neural networks on time scales}
\setcounter{equation}{0} \vspace{1ex} \indent

In this section, we will use Theorem \ref{thm42} to study the existence of weighted pseudo-almost periodic solutions of system \eqref{e1}.

\begin{theorem}\label{thm51}
 Let $u\in\mathbb{U}_{\infty}^{Inv}$. Assume that $(H_1)$-$(H_4)$ and
\begin{itemize}
\item [$(H_5)$] $\min_{1\leq i\leq n}\{\Pi_i\}<\min_{1\leq i\leq n}\{\underline{c_i}\}$ and
there exists a constant $r_0$ such that
$\max_{1\leq i\leq n}\big\{\frac{\eta_i}{\underline{c_i}}\big\}+L\leq r_0$,
where
\[
\eta_i=\sum^n_{j=1}(\overline{a_{ij}}+\overline{b_{ij}})(|f_j(0)|+\alpha_{j}r_0),\quad \Pi_i=\sum^n_{j=1}(\overline{a_{ij}}+\overline{b_{ij}})\alpha_j,
\quad L=\max_{1\leq i\leq
n}\big\{\frac{\overline{I_i}}{\underline{c_i}}\big\},
\]
\[\underline{c_i}=\inf_{t\in\mathbb{T}}
c_i(t),\quad
\overline{c_i}=\sup_{t\in\mathbb{T}}c_i(t),\quad
\overline{I_i}=\sup_{t\in\mathbb{T}}|I_i(t)|,\quad i=1,2,\ldots,n
\]
\end{itemize}
hold, then system \eqref{e1} has a unique weighted pseudo-almost periodic solution in the
region
\[
E=\{\varphi\in PAP(\mathbb{T},\mathbb{R}^n,u):\|\varphi\|_{\infty}\leq r_0\}.
\]
\end{theorem}
\begin{proof}
For any given $\varphi=(\varphi_1,\varphi_2,\ldots,\varphi_n)^T\in E$,
consider the following equation
\begin{equation} \label{d5}
x_i^{\Delta}(t)=-c_i(t)x_i(t)+\sum^n_{j=1}a_{ij}(t)f_j(\varphi_j(t))+\sum_{j=1}^{n}b_{ij}(t)f_j(\varphi_j(t-\gamma_{ij}))+I_i(t),\,i=1,2,\ldots,n
\end{equation}
and its associated homogeneous equation
\begin{equation} \label{d6}
x_i^{\Delta}(t)=-c_i(t)x_i(t),\,\,\,i=1,2,\ldots,n.
\end{equation}
It follows from $(H_3)$ and Lemma \ref{lem42} that \eqref{d6} admits an exponential dichotomy. By Lemma \ref{le51}, we have $$F(t):=(F_1(t),F_2(t),\ldots,F_{n}(t))^T\in PAP(\mathbb{T},\mathbb{R}^{n},u),$$
where
\[
F_i(t)=\sum^n_{j=1}a_{ij}(t)f_j(\varphi_j(t))+\sum^n_{j=1}b_{ij}(t)f_j(\varphi_j(t-\gamma_{ij}))+I_i(t),\,\,\,i=1,2,\ldots,n.
\]
By  Theorem \ref{thm42}, we know that system \eqref{d5} has exactly one weighted pseudo-almost periodic solution
\[
x_{\varphi}(t)=\int^t_{-\infty}X(t)X^{-1}(\sigma(s))F(s)\Delta s=(x_{\varphi_1}(t),\ldots,x_{\varphi_n}(t))^T,
\]
where
\[
x_{\varphi_i}(t)=\int^t_{-\infty}e_{-c_i}(t,\sigma(s))F_i(s)\Delta s,\,\,\,i=1,2,\ldots,n.
\]
Define a nonlinear operator on $E$ by
\[
\Phi(\varphi)(t)=(x_{\varphi_1}(t),\ldots,x_{\varphi_n}(t))^T,\,\,\,\forall \varphi
\in E.
\]
  For any given $\varphi\in
E$, by conditions $(H_1)-(H_5)$, we have
{\setlength\arraycolsep{2pt}
\begin{eqnarray*}
\sup_{t\in\mathbb{T}}|x_{\varphi_i}(t)|&=&\sup_{t\in\mathbb{T}}\bigg|\int^t_{-\infty}e_{-c_i}(t,\sigma(s))
\bigg(\sum\limits^n_{j=1}(a_{ij}(s)f_j(\varphi_j(s))+b_{ij}(s)f_j(\varphi_j(s-\gamma_{ij})))+I_i(s)\bigg)\Delta
s\bigg|\nonumber\\
&\leq&\sup_{t\in\mathbb{T}}\bigg\{\bigg|\int^t_{-\infty}e_{-\underline{c_i}}(t,\sigma(s))
\bigg(\sum\limits^n_{j=1}(\overline{a_{ij}}f_j(\varphi_j(s))+\overline{b_{ij}}f_j(\varphi_j(s-\gamma_{ij})))\bigg)\Delta
s\bigg|\bigg\}+\frac{\overline{I_i}}{\underline{c_i}}\nonumber\\
&\leq&\sup_{t\in\mathbb{T}}\bigg\{\bigg|\int^t_{-\infty}e_{-\underline{c_i}}(t,\sigma(s))
\bigg(\sum\limits^n_{j=1}\overline{a_{ij}}\big(|f_j(0)|+\alpha_j|\varphi_j(s)|\big)\nonumber\\
&&+\sum^n_{j=1}\overline{b_{ij}}\big(|f_j(0)|+\alpha_j|\varphi_j(s-\gamma_{ij})|\big)\bigg)\Delta
s\bigg|\bigg\}+\frac{\overline{I_i}}{\underline{c_i}}\nonumber\\
&\leq&\sup_{t\in\mathbb{T}}\bigg|\int^t_{-\infty}e_{-\underline{c_i}}(t,\sigma(s))
\bigg(\sum\limits^n_{j=1}\big[\overline{a_{ij}}\big(|f_j(0)|+\alpha_{j}r_0\big)+\overline{b_{ij}}\big(|f_j(0)|+\alpha_{j}r_0\big)\big]\bigg)\Delta
s\bigg|+\frac{\overline{I_i}}{\underline{c_i}}\nonumber\\
&\leq&\frac{\eta_i}{\overline{c_i}}+L_1\leq r_0,\,\,\,i=1,2,\ldots,n.
\end{eqnarray*}}
Hence
$
\|\Phi(\varphi)\|_{\infty}=\max_{1\leq i\leq n}\sup_{t\in\mathbb{T}}|x_{\varphi_i}(t)|
\leq r_0.
$
 Therefore, $\Phi(E)\subset E$.

 Taking $\varphi,\psi\in E$ and combining conditions $(H_1)$ and
$(H_5)$, we obtain that
{\setlength\arraycolsep{2pt}
\begin{eqnarray} \label{e9}
&&\sup_{t\in\mathbb{T}}|x_{\varphi_i}(t)-x_{\psi_i}(t)|\nonumber\\
&=&\sup_{t\in\mathbb{T}}
\bigg\{\bigg|\int^t_{-\infty}e_{-c_i}(t,\sigma(s))\bigg(\sum^n_{j=1}
a_{ij}(s)\bigg[f_j(\varphi_j(s))-
f_j(\psi_j(s))\bigg]\nonumber\\
&&+\sum^n_{j=1}b_{ij}(s)\bigg[f_j(\varphi_j(s-\gamma_{ij}))-f_j(\psi_j(s-\gamma_{ij}))\bigg]\bigg)\Delta
s\bigg|\bigg\}\nonumber\\
&\leq&\sup_{t\in\mathbb{T}}\bigg\{\bigg|\int^t_{-\infty}e_{-c_i}(t,\sigma(s))\bigg(\sum^n_{j=1}
a_{ij}(s)\alpha_j|\varphi_j(s)-
\psi_j(s)|\nonumber\\
&&+\sum^n_{j=1}b_{ij}(s)\alpha_j|\varphi_j(s-\gamma_{ij})-\psi_j(s-\gamma_{ij})|\bigg)\Delta s\bigg|\bigg\}\nonumber\\
&\leq&\sup_{t\in\mathbb{T}}\bigg\{\bigg|\int^t_{-\infty}e_{-\underline{c_i}}(t,\sigma(s))\bigg(
\sum^n_{j=1}(\overline{a_{ij}}+\overline{b_{ij}})\alpha_{j}\bigg)\Delta s
\bigg|\bigg\}\|\varphi-\psi\|_{\infty}\nonumber\\
&\leq&\frac{\Pi_i}{\underline{c_i}}\|\varphi-\psi\|_{\infty}
<\|\varphi-\psi\|_{\infty},\,i=1,2,\ldots,n.
\end{eqnarray}}
From \eqref{e9}, we obtain
\begin{equation} \label{e11}
\|\Phi(\varphi)-\Phi(\psi)\|_{\infty}=\max_{1\leq i\leq
n}\sup_{t\in\mathbb{T}}\|x_{\varphi_i}(t)-x_{\psi_i}(t)\|_{\infty}
<\|\varphi-\psi\|_{\infty}.
\end{equation}

 By \eqref{e11}, we see that $\Phi$ is a contraction
mapping from $E$ to $E$. Since $E$ is a closed subset of
$PAP(\mathbb{T},\mathbb{R}^{n},u)$, $\Phi$ has a fixed point in $E$, which means that system
\eqref{e1} has a unique weighted pseudo-almost periodic solution in the region
$
E=\{\varphi\in PAP(\mathbb{T},\mathbb{R}^{n},u):\|\varphi\|_{\infty}\leq r_0\}.
$
This completes the proof.
\end{proof}

\begin{corollary}\label{cl51}
If conditions $(H_1)$-$(H_3)$ and $(H_5)$ hold. Furthermore, assume that $I_i(i=1,2,\ldots,n)$
are almost periodic functions, then  system \eqref{e1} has a unique almost periodic solution in the region
$
E=\{\varphi\in AP(\mathbb{T},\mathbb{R}^{n}):||\varphi||_{\infty}\leq r_0\}.
$
\end{corollary}

\begin{corollary}\label{cl52}
If conditions $(H_1)$-$(H_3)$ and $(H_5)$ hold. Furthermore, assume that $I_i(i=1,2,\ldots,n)$
are pseudo-almost periodic functions, then  system \eqref{e1} has a unique pseudo-almost periodic solution in the region
$
E=\{\varphi\in PAP(\mathbb{T},\mathbb{R}^{n}):||\varphi||_{\infty}\leq r_0\}.
$
\end{corollary}

\section{Exponential stability of the weighted pseudo-almost periodic solution of cellular neural networks on time scales}
\setcounter{equation}{0} \vspace{1ex} \indent

In this section, we derive sufficient conditions for the exponential stability of  weighted pseudo-almost periodic solutions of \eqref{e1}.

\begin{theorem}\label{thm61}
Suppose that $(H_1)$-$(H_5)$ hold, then system \eqref{e1} has a unique weighted pseudo-almost
periodic solution which is globally exponential stable.
\end{theorem}
\begin{proof}
According to Theorem \ref{thm31}, for $u\in\mathbb{U}_{\infty}^{Inv}$, we know that \eqref{e1} has a weighted pseudo-almost periodic
solution
$x^*(t)=\big(x_1^*(t),x_2^*(t),\ldots,x_n^*(t)\big)^{T}$ with the initial value
$\varphi^*(t)=(\varphi^*_1(t),\varphi^*_2(t),\ldots,\varphi^*_n(t))^T$.
Suppose that
$x(t)=\big(x_1(t),x_2(t),\ldots,x_n(t)\big)^{T}$
is an arbitrary solution of \eqref{e1}  with the initial value
$\varphi(t)=(\varphi_1(t),\varphi_2(t),\ldots,\varphi_n(t))^T$. Then it follows from system \eqref{e1}
that
{\setlength\arraycolsep{2pt}
\begin{eqnarray} \label{e12}
&&u_i^\Delta(s)+c_i(s)u_i(s)\nonumber\\
&=&\sum^n_{j=1}a_{ij}(s)\bigg[f_j(u_j(s)
+y_j^*(s))-f_j(y_j^*(s))\bigg]\nonumber\\
&&+\sum^n_{j=1}b_{ij}(s)\bigg[f_j(u_j(s-\gamma_{ij})+y_j^*(s-\gamma_{ij}))-f_j(y_j^*(s-\gamma_{ij}))\bigg],
\end{eqnarray}}
where $u_i(s)=x_i(s)-x_i^*(s)$ and
$i=1,2,\ldots,n$, the initial condition of \eqref{e12} are
\[
\psi_i(s)=\varphi_i(s)-\varphi_i^*(s),\,\,\,s\in[-\gamma,0]_{\mathbb{T}},\,\,\,i=1,2,\ldots,n.
\]
 Let $H_i$ and $\overline{H_j}$ be defined by
\[
H_i(\epsilon)=\underline{c_i}-\epsilon-\sum^n_{j=1}\alpha_j\big(\overline{a_{ij}}\exp(\epsilon\sup_{s\in\mathbb{T}}\mu(s))+
\overline{b_{ij}}\exp\big(\epsilon(\gamma
+\sup_{s\in\mathbb{T}}\mu(s))\big)\big)
,
\]
where $i=1,2,\ldots,n,\epsilon\in[0,+\infty)$. By $(H_5)$, we
get
\[
H_i(0)=\underline{a_i}-\sum^n_{j=1}(\overline{a_{ij}}+\overline{b_{ij}})\alpha_j=\underline{c_i}-\Pi_i>0,\quad i=1,2,\ldots,n.
\]
Since $H_i, i=1,2,\ldots,n$ are continuous on $[0,+\infty)$ and
$H_i(\epsilon)\rightarrow-\infty$
as $\epsilon\rightarrow+\infty$, there exist
$\epsilon_i>0$ such that
$H_i(\epsilon_i)=0$ and
$H_i(\epsilon)>0$ for $\epsilon\in(0,\epsilon_i)$. By choosing
$\varepsilon=\min\limits_{1\leq i\leq n}\{\epsilon_i\}$,
we have
$
H_i(\varepsilon)\geq 0,  i=1,2,\ldots,n.
$
So, we can choose a positive constant
$0<\lambda<\min\big\{\varepsilon,\min\limits_{1\leq i\leq
n}\{\underline{c_i}\}\big\}$ such that
$
H_i(\lambda)>0,  i=1,2,\ldots,n,
$
which implies that
\begin{equation} \label{e14}
\frac{1}{\underline{c_i}-\lambda}\bigg[\sum^n_{j=1}\overline{a_{ij}}\alpha_j\exp(\lambda\sup_{s\in\mathbb{T}}\mu(s))+
\sum^n_{j=1}\overline{b_{ij}}\alpha_j\exp\big(\lambda(\gamma+
\sup_{s\in\mathbb{T}}\mu(s))\big)\bigg]<1,\,\,\,i=1,2,\ldots,n.
\end{equation}
 Multiplying \eqref{e12} by $e_{-c_i}(t_0,\sigma(s))$ and integrating
on $[t_0,t]_{\mathbb{T}}$, for $i=1,2,\ldots,n$, we obtain
{\setlength\arraycolsep{2pt}
\begin{eqnarray} \label{e16}
u_i(t)&=&u_i(t_0)e_{-c_i}(t,t_0)+\int^t_{t_0}e_{-c_i}(t,\sigma(s))\bigg(\sum^n_{j=1}a_{ij}(s)\bigg[f_j(u_j(s)+y_j^*(s))-f_j(y_j^*(s))\bigg]\nonumber\\
&&+\sum^n_{j=1}b_{ij}(s)\bigg[f_j(u_j(s-\gamma_{i j})
+y_j^*(s-\gamma_{i j}))-f_j(y_j^*(s-\gamma_{i j})) \bigg]\bigg)\Delta
s.
\end{eqnarray}}
Let $M=\max\limits_{1\leq i\leq n}\big\{\frac{\underline{c_i}}{\sum\limits^n_{j=1}(\overline{a_{ij}}+\overline{b_{ij}})\alpha_j}\big\}
$, by $(H_5)$ we have $M>1$. Thus
\[
\frac{1}{M}-\frac{1}{\underline{c_i}-\lambda}\bigg[\sum^n_{j=1}\overline{a_{ij}}\alpha_j\exp(\lambda\sup_{s\in\mathbb{T}}\mu(s))
+\sum\limits^n_{j=1}\overline{b_{ij}}\alpha_j
\exp\big(\lambda(\gamma+\sup_{s\in\mathbb{T}}\mu(s))\big)\bigg]\leq
0.
\]
 It is easy to see that
\[
|u_i(t)|=|\psi_i(t)|\leq\|\psi\|_{\infty}\leq M
e_{\ominus\lambda}(t,t_0)\|\psi\|_{\infty},\,\,\,t\in
[-v,0]_{\mathbb{T}},\,\,\,i=1,2,\ldots,n,
\]
where $\lambda\in \mathcal{R}^+$ is the same as that in \eqref{e14}, which implies that
\[
\|x(t)-x^*(t)\|=\max_{1\leq i\leq n}\big\{|u_i(t)|\big\}\leq M
e_{\ominus\lambda}(t,t_0)\|\psi\|_{\infty},\,\,\,t\in
[-v,0]_{\mathbb{T}}.
\]
Next, we claim that
\begin{equation} \label{e18}
\|x(t)-x^*(t)\|\leq M
e_{\ominus\lambda}(t,t_0)\|\psi\|_{\infty},\,\,\,\forall
t\in(0,+\infty)_{\mathbb{T}}.
\end{equation}
In order to prove \eqref{e18}, we first show for any $p>1$, the following inequality holds
\begin{equation} \label{e19}
\|x(t)-x^*(t)\|< pM
e_{\ominus\lambda}(t,t_0)\|\psi\|_{\infty},\,\,\,\forall
t\in(0,+\infty)_{\mathbb{T}}.
\end{equation}
If \eqref{e19} is not true, then there must be some
$t_1\in(0,+\infty)_{\mathbb{T}},\,C>1$ and some $k$ such that
\begin{equation} \label{e20}
\|x(t_1)-x^*(t_1)\|=|x_k(t_1)-x_k^*(t_1)|=CpM
e_{\ominus\lambda}(t_1,t_0)\|\psi\|_{\infty}
\end{equation}
and
\begin{equation} \label{e21}
\|x(t)-x^*(t)\|\leq CpM
e_{\ominus\lambda}(t,t_0)\|\psi\|_{\infty},\quad\forall
t\in[-v,t_1]_{\mathbb{T}}.
\end{equation}
 By \eqref{e16}-\eqref{e21} and $(H_2)$-$(H_5)$, we obtain
{\setlength\arraycolsep{2pt}
\begin{eqnarray} \label{e22}
|u_i(t_1)|&\leq&
e_{-c_i}(t_1,t_0)\|\psi\|_{\infty}+\int_{t_0}^{t_1}CpM\|\psi\|_{\infty}
e_{-c_i}(t_1,\sigma(s))
\bigg(\sum^n_{j=1}\overline{a_{ij}}\alpha_{j}e_{\ominus\lambda}(s,t_0)\nonumber\\
&&+\sum^n_{j=1}\overline{b_{ij}}\alpha_{j}e_{\ominus\lambda}(s-\gamma_{ij},t_0)
\bigg)\Delta
s\nonumber\\
&\leq& CpM
e_{\ominus\lambda}(t_1,t_0)\|\psi\|_{\infty}\bigg\{\frac{1}{CpM}e_{-c_i}(t_1,t_0)
e_{\ominus\lambda}(t_0,t_1)+\int_{t_0}^{t_1}e_{-c_i}(t_1,\sigma(s))e_{\lambda}(t_1,\sigma(s)))\nonumber\\
&&\times\bigg(\sum^n_{j=1}\overline{a_{ij}}\alpha_{j}e_{\ominus\lambda}(s,\sigma(s))+\sum^n_{j=1}\overline{b_{ij}}\alpha_{j}e_{\ominus\lambda}(s-\gamma,\sigma(s))\bigg)\Delta
s\bigg\}\nonumber\\
&<& CpM e_{\ominus\lambda}(t_1,t_0)\|\psi\|_{\infty}
\bigg\{\frac{1}{M}e_{-c_i\oplus\lambda}(t_1,t_0)+\bigg(\sum^n_{j=1}\overline{a_{ij}}\alpha_{j}\exp(\lambda\sup_{s\in\mathbb{T}}\mu(s))\nonumber\\
&&+\sum\limits^n_{j=1}\overline{b_{ij}}\alpha_j\exp\big(\lambda(\gamma+
\sup_{s\in\mathbb{T}}\mu(s))\big)\bigg)\int_{t_0}^{t_1}e_{-c_i\oplus\lambda}(t_1,\sigma(s))\Delta
s\bigg\}\nonumber\\
&\leq& CpM e_{\ominus\lambda}(t_1,t_0)\|\psi\|_{\infty}
\bigg\{\frac{1}{M}e_{-c_i\oplus\lambda}(t_1,t_0)+\bigg(\sum^n_{j=1}\overline{a_{ij}}\alpha_j\exp(\lambda\sup_{s\in\mathbb{T}}\mu(s))\nonumber\\
&&+\sum\limits^n_{j=1}\overline{b_{ij}}\alpha_j\exp\big(\lambda(\gamma+
\sup_{s\in\mathbb{T}}\mu(s))\big)\bigg)\frac{1-e_{-c_i\oplus\lambda}(t_1,t_0)}{\underline{c_i}-\lambda}\bigg\}\nonumber\\
&\leq& CpM e_{\ominus\lambda}(t_1,t_0)\|\psi\|_{\infty}
\bigg\{\bigg[\frac{1}{M}-\frac{1}{\underline{c_i}-\lambda}
\bigg(\sum^n_{j=1}\overline{a_{ij}}\alpha_j\exp(\lambda\sup_{s\in\mathbb{T}}\mu(s))\nonumber\\
&&+\sum\limits^n_{j=1}\overline{b_{ij}}\alpha_j\exp\big(\lambda(\gamma+
\sup_{s\in\mathbb{T}}\mu(s))\big)\bigg)\bigg]e_{-c_i\oplus\lambda}(t_1,t_0)\nonumber\\
&&+\frac{1}{\underline{c_i}-\lambda}\bigg(\sum^n_{j=1}\overline{a_{ij}}\alpha_j\exp(\lambda\sup_{s\in\mathbb{T}}\mu(s))
+\sum\limits^n_{j=1}\overline{b_{ij}}\alpha_j\exp\big(\lambda(\gamma+
\sup_{s\in\mathbb{T}}\mu(s))\big)\bigg)\bigg\}\nonumber\\
 &<& CpM e_{\ominus\lambda}(t_1,t_0)\|\psi\|_{\infty}.
\end{eqnarray}}
\eqref{e22} implies that
\[
|x_k(t_1)-x_k^*(t_1)|<CpM
e_{\ominus\lambda}(t_1,t_0)\|\psi\|_{\infty},\,\,\,\forall
k\in\{1,2,\ldots,n\},
\]
 which contradicts  \eqref{e20}, and so \eqref{e19} holds. Letting $p\rightarrow 1$, then \eqref{e18} holds. Hence,
the weighted pseudo-almost periodic solution of system \eqref{e1} is globally
exponentially stable. The global  exponential stability implies that
the weighted pseudo-almost periodic solution is unique.
\end{proof}

\begin{corollary}
If conditions $(H_1)$-$(H_3)$ and $(H_5)$ hold. Furthermore, suppose that $I_i(i=1,2,\ldots,n)$
are almost periodic functions, then system \eqref{e1} has a  unique almost periodic solution which is globally exponential stable.
\end{corollary}
\begin{corollary}
If conditions $(H_1)$-$(H_3)$ and $(H_5)$ hold. Furthermore, suppose that $I_i(i=1,2,\ldots,n)$
are pseudo-almost periodic functions, then system \eqref{e1} has a  unique pseudo-almost periodic solution which is globally exponential stable.
\end{corollary}

\section{Numerical examples}\setcounter{equation}{0}
\vspace{1ex} \indent

Consider the following neural network:
\begin{equation} \label{e24}
x_i^{\Delta}(t)=-c_i(t)x_i(t)+\sum^2_{j=1}a_{ij}(t)f_j(x_j(t))+\sum^2_{j=1}b_{ij}(t)f_j(x_j(t-\gamma_{ij}))+I_i(t),\,i=1,2,
\end{equation}
where $
f_1(x)=\frac{\cos^{3}x+5}{18},
f_2(x)=\frac{\cos^{3}x+3}{12}$
and the weight $u=\frac{1}{2}+e^{-|t|}$.\\

\begin{example} Take $\mathbb{T}=\mathbb{R}$ and
\[
c_1(t)=11+|\cos(\sqrt{2}t)|,\quad c_2(t)=12-|\sin
t|,
\]
\[
I_1(t)=\frac{2}{16}(\sin t+\sin(t+\frac{\pi}{6})),\quad I_2(t)=\frac{\sqrt{2}}{8}(\sin t+\sin(\frac{\pi}{4}+\sqrt{2}t)),
\]
\[
\begin{split}
 (a_{ij}(t))_{2\times 2}=
 \left(\begin{array}{ccc}
  \frac{15}{7}|\cos t| &&  \frac{10}{7}|\sin t| \\
\frac{18}{7}|\cos t| &&  \frac{13}{14}|\sin t|
\end{array}\right)
\end{split},\,\,
\begin{split}
 (b_{ij}(t))_{2\times 2}=
 \left(\begin{array}{ccc}
 2|\sin t| &&  \frac{5}{3}|\cos t| \\
\frac{10}{3}|\sin t| &&  \frac{1}{24}|\sin t|
\end{array}\right)
\end{split}.
\]
Let $\gamma_{ji}(i,j=1,2)$ be real numbers, then $(H_2)-(H_4)$ hold. Let
$\alpha_1=\alpha_2=\frac{1}{4}$, then $(H_1)$ holds.
Next, let us check $(H_5)$, if we take $r_0=1$, then
\[
\max\big\{\frac{\eta_1}{\underline{c_1}},
\frac{\eta_2}{\underline{c_2}}\big\}+L=
\frac{1064}{2772}+\frac{\sqrt{2}}{44}\approx0.566<1=r_0
\]
and
\[
\max\{\Pi_1,\Pi_2\}=\max\big\{\frac{152}{84},\frac{165}{96}\big\}
=\frac{152}{84}<11=\min\{\underline{c_1},\underline{c_2}\}.
\]
Thus, $(H_5)$ holds for $r_0=1$. Now, by Theorem \ref{thm51} and Theorem \ref{thm61}, system \eqref{e24} has a unique weighted pseudo-almost periodic solution in the region
$
E=\{\varphi\in PAP(\mathbb{T},\mathbb{R}^2,u):\|\varphi\|_{\infty}\leq 1\},
$
which is globally exponential stable.
\end{example}

\begin{example} Take $\mathbb{T}=\mathbb{Z}$ and
\[
c_1(t)=0.9-0.1|\sin(\sqrt{3}t)|,\,\,
c_2(t)=0.8+0.1\cos^2t,
\]
\[
I_1(t)=\frac{1}{32}(3\sin t+\sqrt{3}\cos t),\quad
I_2(t)=\frac{1}{64}(\sqrt{2}\sin t+\sqrt{2}\cos t+2\sin t),
\]
\[
\begin{split}
 (a_{ij}(t))_{2\times 2}=
 \left(\begin{array}{ccc}
  \frac{1}{7}|\sin t| &&  \frac{1}{7}\sin^{2}t \\
\frac{3}{14}|\cos t| &&  \frac{1}{14}|\sin(\sqrt{2}t)|
\end{array}\right)
\end{split},\,\,
\begin{split}
 (b_{ij}(t))_{2\times 2}=
 \left(\begin{array}{ccc}
  \frac{1}{8}|\sin t| &&  \frac{1}{24}\cos^{2}t\\
\frac{1}{48}|\sin t| &&  \frac{1}{16}|\cos t|
\end{array}\right)
\end{split}.
\]
Let $\gamma_{ji}(i,j=1,2)$ be arbitrary nature numbers, then $(H_2)$-$(H_4)$ hold. Let
$\alpha_1=\alpha_2=\frac{1}{4}$, then $(H_1)$ holds.
Next, let us check $(H_4)$, if we take $r_0=1$, then
\[
\max\big\{\frac{\eta_1}{\underline{c_1}},
\frac{\eta_2}{\underline{c_2}}\big\}+L=
\frac{665}{2016}+\frac{5}{16}\approx0.642<1=r_0
\]
and
\[
\max\{\Pi_1,\Pi_2\}=\max\big\{\frac{19}{168},\frac{31}{336}\big\}
=\frac{19}{168}<0.8=\min\{\underline{c_1},\underline{c_2}\}.
\]
Thus, $(H_5)$ holds for $r_0=1$. Now, by Theorem \ref{thm51} and Theorem \ref{thm61}, system \eqref{e24} has a unique weighted pseudo-almost periodic solution in the region
$
E=\{\varphi\in PAP(\mathbb{T},\mathbb{R}^2,u):\|\varphi\|_{\infty}\leq 1\},
$
which is globally exponential stable.
\end{example}

\noindent\textbf{Conflict of Interests}\\

The authors declare that there is no conflict of interests
regarding the publication of this paper.

\end{document}